\documentclass[a4paper,12pt]{article}
\usepackage[centertags]{amsmath}
\usepackage{amsfonts}
\usepackage{amssymb}
\usepackage{amsthm}
\usepackage{amsmath}
\usepackage{dsfont}
\usepackage{graphicx}
\usepackage{tikz, subfigure}
\usepackage{pstricks-add}
\usepackage{verbatim}

\usepackage[latin1]{inputenc}
\usepackage{tikz}
\definecolor {processblue}{cmyk}{0.96,0,0,0}
\usepackage{amsfonts,graphicx,amsmath,amssymb,hyperref,color}
\usepackage[english]{babel}
\usepackage{authblk}

\newtheorem{Theorem}{Theorem}[section]

\newtheorem{Lemma}{Lemma}[section]
\newtheorem{Corollary}{Corollary}[section]

\newtheorem{Example}{Example}[section]

\AtEndDocument{\bigskip{\footnotesize%
  \textsc{Department of Mathematics, Shanghai Key Laboratory of PMMP, East China Normal University, Shanghai 200241,
People's Republic of China} \par
  \textit{E-mail address:} \texttt{	
chenxiu1216@163.com.} \par
}}
\AtEndDocument{\bigskip{\footnotesize%
  \textsc{Department of Mathematics, Ningbo University, Ningbo, Zhejiang, People's Republic of China} \par
  \textit{E-mail address:} \texttt{kanjiangbunnik@yahoo.com} \par
}}
\AtEndDocument{\bigskip{\footnotesize%
  \textsc{Department of Mathematics, Shanghai Key Laboratory of PMMP, East China Normal University, Shanghai 200241,
People's Republic of China} \par
  \textit{E-mail address:} \texttt{wxli@math.ecnu.edu.cn.} \par
}}

\usepackage{lipsum}

\makeatletter
\newcommand*{\rom}[1]{\expandafter\@slowromancap\romannumeral #1@}
\makeatother
\begin{document}

\title{xxxx}
\date{}
 \title{Estimating the Hausdorff dimensions of  univoque sets  for self-similar sets}
 \author{Xiu Chen, Kan Jiang\thanks{Corresponding author  }  and Wenxia Li}
\maketitle{}

\begin{abstract}
An approach is given for estimating the Hausdorff dimension of the univoque set   of a self-similar set.
This sometimes allows us to get the exact Hausdorff dimensions of the univoque sets.



\end{abstract}

\section{Introduction}

  Let  $\{f_{i}\}_{i=1}^{m}$ be an iterated function system (IFS) of contractive similitudes on $\mathbb{R}^{d}$ defined as
 \begin{equation*}\label{Hutchinson formula}
 f_{i}(x)=r_{i}R_{i}x+b_{i},\;\; i\in \Omega =\{1,\ldots ,m\},
\end{equation*}
where $0<r_{i}<1$ is the contractive ratio, $R_{i}$ is an orthogonal transformation and $b_{i}\in \mathbb{R}^{d}.$ Then there exists a unique nonempty compact set $K\subseteq \mathbb{R}^{d}$ satisfying (cf. \cite{Hutchinson})
\begin{equation}\label{sss}
K=\bigcup _{i=1}^m f_i(K).
\end{equation}
The set $K$ is called  the self-similar set generated by the IFS
$\{f_{j}\}_{j=1}^{m}$.
The IFS $\{f_{j}\}_{j=1}^{m}$ is said to  satisfy the open set condition (OSC) (cf. \cite{Hutchinson})  if there exists a non-empty bounded open set $V\subseteq \mathbb{R}^{d}$ such that
$$
V\supseteq \bigcup _{i=1}^mf_i(V)\; \textrm{with disjoint union on the right side}.
$$
Under the open set condition, the Hausdorff dimension of $K$ coincides with  the similarity dimension, denoted by $\dim_SK$, which is the unique solution $s$ of the equation $\sum_{j=1}^{m}r_j^s=1$.
For any  $x \in K$, there exists a sequence
$(i_n)_{n=1}^{\infty}\in\{1,\ldots,m\}^{\mathbb{N}}$ such that
\[x=\lim_{n\to \infty}f_{i_1}\circ \cdots\circ f_{i_n}(0)=\bigcap_{n=1}^{\infty}f_{i_1}\circ \cdots\circ f_{i_n}(K).\]
  Such  sequence $(i_n)_{n=1}^{\infty}$ is called a coding of $x$. The attractor $K$ defined by (\ref{sss}) may equivalently be defined to be the set
of points in $\mathbb{R}^d$ which admit a coding, i.e.,
one can define a surjective projection map between the symbolic space $\{1,\ldots, m\}^{\mathbb{N}}$ and the self-similar set $K$ by
 $$\Pi((i_n)_{n=1}^{\infty}):=\lim_{n\to \infty}f_{i_1} \circ\cdots \circ
f_{i_n}(0).$$
A point  $x\in K$ may
have multiple codings.
$x\in K$  is called a
univoque point if it has only one coding. The set of univoque points is called the  univoque
set,  denoted  by $U$ or $U_1$. Generally, for $k\in \mathbb N$ we set
\begin{equation*}\label{kset}
U_k=\{x\in K: x \;\textrm{has exact }\; k\; \textrm{codings}\}.
\end{equation*}

 The univoque set plays a pivotal  role in  studying the sets of  multiple codings (cf. \cite{KarmaKan, DJKL, KarmaKan2}), e.g., we have
\begin{equation}\label{ksetine}
\dim_{H}U_{k}\leq \dim_{H}U\;\;\textrm{for}\; k\geq2,
 \end{equation}
since
 $U_{k}\subseteq \bigcup_{{\bf i}\in \Omega ^*}f_{\bf i}(U)$ where, as usual, $\Omega ^*=\bigcup _{n=1}^\infty \Omega ^n$.
Therefore, it is crucial to find the Hausdorff dimension of the univoque set for  self-similar sets. There are many papers about the Hausdorff dimension of  $U$ when $K$ is an interval (cf. \cite{SKK, GS, MK,  KKLL, KDLW, BakerG, SN, KO, Sidorov}).

In the present paper, we offer an approach to estimate $\dim _HU$ for general self-similar sets.
 Let $M$ be a nonempty compact subset of $\mathbb{R}^{d}$ satisfying $f_i(M)\subseteq M$ for $1\leq i\leq m$ (so $K\subseteq M$). Let
$$
{\mathcal S} _1=\{{\bf k}\in \Omega:f_{\bf k}(M)\cap f_{\bf j}(M)=\emptyset \,\,\mbox{for all }\,\,{\bf j}\in \Omega \setminus \{\bf k\} \}\;\textrm{and}\; {\mathcal T}_{1}=\Omega \setminus {\mathcal S} _{1}.
$$
For positive integer $i$ let
 \begin{equation}\label{Sdef}
 \begin{split}
 &{\mathcal S} _{i+1}=\{\textbf{k}\in {\mathcal T}_{i}\times \Omega: f_{\bf k}(M)\cap f_{\bf j}(M)=\emptyset \,\,\mbox{for all }\,\,{\bf j}\in ({\mathcal T}_{i}\times \Omega) \setminus \{\textbf{k}\} \},\\
  &{\mathcal T}_{i+1}=({\mathcal T}_i\times \Omega )\setminus {\mathcal S} _{i+1}.
  \end{split}
 \end{equation}
Note that ${\mathcal S}_i$ may be empty for some $i$. Let
\begin{equation}\label{gammaoneway}
\Gamma =\bigcup _{i\geq 1}{\mathcal S} _i.
\end{equation}
It is clear that   $\Gamma $  becomes
    largest when $M$ is taken as $K$.
An ${\bf i}\in \Omega ^{\mathbb N}$ is said to begin with $\Gamma $ if ${\bf i}|k\in \Gamma $ for some $k\in \mathbb N$. Let
\begin{equation}\label{setV}
V=\{{\bf i}\in \Omega ^{\mathbb N}: {\bf i}\;\textrm{does not begin with}\; \Gamma \}.
\end{equation}
In this paper we obtain
\begin{Theorem}\label{thm1}
Let $\Gamma $ and $V$ be defined by (\ref{gammaoneway}) and (\ref{setV}) respectively.
Then
$$
\dim _HU=\max \{\dim _H\Pi \left (\Gamma ^{\mathbb N}\right ), \dim _H\Pi (V\cap \Pi^{-1}(U))\}.
$$
\end{Theorem}
Let $s$ be determined by
\begin{equation*}\label{svalue}
 \sum _{{\bf i}\in \Gamma }r_{\bf i}^{s}=1.
 \end{equation*}
 Then we have  $\dim _H\Pi \left (\Gamma ^{\mathbb N}\right )=s$ which will be proved in Lemma \ref{gamma1}. Hence
 \setcounter{Corollary}{1}
\begin{Corollary}\label{coao}
We have  $\dim _HU\geq s$ and the equality holds if and only if $\dim _H\Pi (V\cap \Pi^{-1}(U))\leq s$.
\end{Corollary}
The OSC plays an important role in determining the Hausdorff dimension of a self-similar set. Let us recall that $K$ is generated by the IFS $\{f_i\}_{i=1}^{m}$ in (\ref{sss}). The following fact is obvious:
\begin{equation}\label{Udim}
0<\mathcal{H}^s(U)=\mathcal{H}^s(K)<\infty\;\; \textrm{if}\;  \{f_i\}_{i=1}^{m}\;\textrm{satisfies the OSC},
\end{equation}
where $s$ is given by $\sum _{i=1}^mr_i^s=1$. In fact, we have  $U= K\setminus \left (K^*\cup \bigcup _{{\bf i}\in \Omega ^*} f_{\bf i}(K^*)\right )$
with $K^*=\bigcup_{i\neq j}(f_i(K)\cap f_{j}(K))$ and the OSC implies that  $\mathcal{H}^s(f_i(K)\cap f_{j}(K))=0$ for any  $i\neq j$ (see \cite{Schief}).

From (\ref{Udim}) it follows that $\dim_{H}U=\dim_{H}K=\dim_{S}K$ if the IFS $\{f_i\}_{i=1}^{m}$
  satisfies the open set condition. We shall show that under some extra condition the inverse is also true. An IFS $\{f_i\}_{i=1}^{m}$  is said to have
  an exact overlap if there exist distinct ${\bf i}, {\bf j}\in \Omega ^*$ such that $f_{\bf i}=f_{\bf j}$.  The notion of ``general finite type" appeared
  in the following Lemma  which was posed by Lau and Ngai
  in \cite{LauNgai}. We have
\setcounter{Theorem}{2}
\begin{Theorem}\label{Main1}
Let $K$ be the self-similar set generated by the IFS $\{f_i\}_{i=1}^{m}$. Suppose that $\{f_i\}_{i=1}^{m}$ is of
 general finite type.  Then $\{f_i\}_{i=1}^{m}$ satisfies the open set condition if and only if $\dim_{H}U=\dim_{S}K.$
\end{Theorem}

This paper is organized as follows. In section 2, we give the proofs of Theorems \ref{thm1} and \ref{Main1}.
The section 3 is devoted to some examples.

\section{Proof of  Theorems \ref{thm1} and \ref{Main1}}
 Denote by ${\bf i}{\bf j}$ the concatenation of  ${\bf i}, {\bf j}\in \Omega ^*$ and ${\bf i}^k$ stands for the  concatenation of  ${\bf i}$ with itself
 $k$ times.   By $|{\bf i}|$ we denote the length of ${\bf i}\in \Omega ^*$.
   For ${\bf i}=i_1\cdots i_k\in \Omega ^*$ we denote by $[\bf i ]$ the cylinder set based on $\bf i$, i.e.,
  $[{\bf i}]=\{(x_i)\in \Omega ^{\mathbb N}:\; x_i=i_i\;\textrm{for}\; 1\leq i\leq k\}$. For an ${\bf i}=(i_k)_{k\geq 1}\in \Omega ^{\mathbb N}$ let ${\bf i}|p=i_1\cdots i_p$. For ${\bf i}=i_1\cdots i_k\in \Omega ^*$ denote $f_{\bf i}=f_{i_1}\circ \cdots \circ f_{i_k}$ and $r_{\bf i}=\prod _{\ell =1}^kr_{i_\ell}$.

  \begin{Lemma}\label{gamma1}
   Let $\Gamma \subseteq \Omega ^*$ be given by (\ref{gammaoneway}). Then
    $\Pi \left (\Gamma ^{\mathbb N}\right )\subseteq U$ and $\dim _H \Pi \left (\Gamma ^{\mathbb N}\right )=s$ where
   $s$ is determined by  $ \sum _{{\bf i}\in \Gamma }r_{\bf i}^{s}=1.$
  \end{Lemma}
\begin{proof}
Note that  by the definition of $\Gamma $ we have

(I) $[{\bf i}], {\bf i}\in \Gamma $ are pairwise disjoint;

   (II) $f_{\bf i}(K)\cap \Pi\left (\Omega ^{\mathbb N}\setminus [{\bf i}]\right )=\emptyset $ for each ${\bf i}\in \Gamma $.

First we show that $\Pi \left (\Gamma ^{\mathbb N}\right )\subseteq U$.
For an $x\in \Pi \left (\Gamma ^{\mathbb N}\right )$ let $x=\Pi((x_k)_{k\geq 1})$ with $(x_k)_{k\geq 1}\in \Gamma ^{\mathbb N}$. Suppose that $(y_k)_{k\geq 1}\in \Omega ^{\mathbb N}$ satisfies that
$x=\Pi((y_k)_{k\geq 1})$. We claim that $(y_k)_{k\geq 1}=(x_k)_{k\geq 1}$. On the contrary, let $\ell $ be the smallest integer such that $y_\ell \neq x_\ell$. Let $(x_k)_{k\geq 1}=({\bf i}_k)_{k\geq1 }$ with ${\bf i}_k\in \Gamma $.
Let $\gamma $ be smallest integer such that $\ell \leq |{\bf i}_1\cdots {\bf i}_\gamma |$. Then
$$
\Pi ((x_k)_{k\geq \delta })=\Pi ((y_k)_{k\geq \delta })\; \textrm{where}\; \delta =|{\bf i}_1\cdots {\bf i}_\gamma |-
|{\bf i}_\gamma |+1.
$$
However, $\Pi ((x_k)_{k\geq \delta })\in f_{{\bf i}_\gamma }(K), {\bf i}_\gamma \in \Gamma $ and
$(y_k)_{k\geq \delta }\notin [{\bf i}_\gamma ]$. This leads a contradiction to the fact $f_{{\bf i}_\gamma }(K)\cap \Pi\left (\Omega ^{\mathbb N}\setminus [{\bf i}_\gamma ]\right )=\emptyset $.

In what follows we prove that $\dim _H \Pi \left (\Gamma ^{\mathbb N}\right )=s$. If $\Gamma $ is finite, then $\Pi \left (\Gamma ^{\mathbb N}\right )$ is a self-similar set generated by the IFS $\{f_{\bf i}: \;{\bf i}\in \Gamma \}$. This IFS satisfies the OSC since $K\supseteq \bigcup _{{\bf i}\in \Gamma }f_{\bf i}(K)$ with disjoint union. Thus $\dim _H\Pi \left (\Gamma ^{\mathbb N}\right )=s$.

In the following we assume that $\Gamma $ is infinite. Denote $\Gamma _{k}=\{{\bf i}\in \Gamma : |{\bf i}|\leq k\}$, $k\in \mathbb N$. Then $\Gamma _{ k}$ is finite (we assume $k$ is big enough such that $\Gamma _{k}\ne \emptyset$). Thus
$$
\dim _H\Pi \left (\Gamma _{ k}^{\mathbb N}\right )=s_k\;\;\textrm{where}\; \sum _{{\bf i}\in \Gamma _{ k}}r_{\bf i}^{s_k}=1.
$$
Therefore, $\dim _H\Pi \left (\Gamma ^{\mathbb N}\right )\geq \sup _ks_k=\lim _{k\rightarrow \infty}s_k=s$ where the last equality can be obtained by the equation $ \sum _{{\bf i}\in \Gamma }r_{\bf i}^{s}=1$ and $\Gamma =\bigcup _{k\geq 1}\Gamma _{ k}$.

Arbitrarily fix a $t>s$. For any $\delta >0$ one can take a big integer $n$ such that each set in $\{f_{\bf i}(K): {\bf i}\in \Gamma ^n\}$ has diameter less that $\delta $. Note that
$$
\sum _{{\bf i}\in \Gamma ^n}|f_{\bf i}(K)|^t=|K|^t\left (\sum _{{\bf i}\in \Gamma }r_{\bf i}^{t}\right )^n\leq |K|^t,
$$
which implies that $\dim _H\Pi \left (\Gamma ^{\mathbb N}\right )\leq t$.
\end{proof}

\begin{proof}[Proof of Theorem \ref{thm1}]
Note that
\begin{equation*}
\begin{split}
\Omega ^{\mathbb N}&=V\cup V^c=V\cup \Gamma ^{\mathbb N}\cup \{{\bf u}{\bf i}: {\bf u}\in \Gamma ^*,\; {\bf i}\in V\}.
\end{split}
\end{equation*}
Thus
$$
\Pi^{-1}(U)=\Gamma^{\mathbb N}\cup (V\cap \Pi^{-1}(U))\cup \{{\bf u}{\bf j}: {\bf u}\in \Gamma ^*, \;{\bf j}\in V\cap \Pi^{-1}(U)\},
$$
which implies the desired result.
\end{proof}
Now we turn to proving  Theorem \ref{Main1}. We need following
\begin{Lemma}\cite[Theorem 2.1]{ZYL}\label{ZYL}
An IFS
 $\{f_i\}_{i=1}^{m}$ satisfies the open set condition if and only if it is of general finite type and has no exact overlaps.
\end{Lemma}
\begin{proof}[Proof of Theorem \ref{Main1}]
The necessity follows from (\ref{Udim}). We now prove the sufficiency. Note that $\{f_i\}_{i=1}^{m}$ is of
 general finite type. Thus by Lemma \ref{ZYL} it suffices to show that the IFS $\{f_i\}_{i=1}^{m}$ has no  exact overlaps.
 Otherwise, there exist distinct ${\bf i}, {\bf j}\in \Omega ^*$ such that $f_{\bf i}=f_{\bf j}$. Let $K_1$ be the self-similar set generated
 by the IFS $\{f_{\bf k}: {\bf k}\in \Omega ^{|{\bf i}|}\;\;\textrm{and}\;\; {\bf k}\ne {\bf i}\}$. Then $\dim _HK_1\leq \dim _HK< \dim _SK$.
 On the other hand, for any $x\in U$ its unique coding cannot contain the block ${\bf i}$ and so $x\in K_1$. Thus, $\dim _HU\leq  \dim _HK_1\leq \dim _HK< \dim _SK$, a contradiction!
\end{proof}

\section{Examples}

The result in the following example  was obtained in \cite{ZL} by giving a lexicographical characterization of the unique codings.
Now we  reprove it by applying Theorem \ref{thm1}, which provides a quite different way from that in   \cite{ZL}.

 \begin{Example}\label{EX1}(see \cite{ZL})
Let $K$ be the self-similar set generated by the IFS
$$
\left\{f_{1}(x)=\rho x ,\; f_{2}(x)=\rho x+\rho,\;f_{3}(x)=\rho x +1\right\} \;\textrm{where}\; 0<\rho<(3-\sqrt{5})/2
.$$
Then $\dim_{H}U=\frac{\log \lambda}{-\log \rho},$  where $\lambda\approx 2.3247$ is the appropriate solution of $$x^{3}-3x^{2}+2x-1=0.$$
\end{Example}
\begin{proof}
First one can check that $f_1\circ f_3=f_2\circ f_1$.
 Take $M=[0,(1-\rho)^{-1}]$. Then
$$
f_1(M)\cap f_2(M)=[0, \rho /(1-\rho)]\cap [\rho, (2\rho-\rho^2)/(1-\rho )]=[\rho, \rho /(1-\rho)]
$$
and
$$
f_1(M)\cap f_3(M)= f_2(M)\cap f_3(M)=\emptyset.
$$
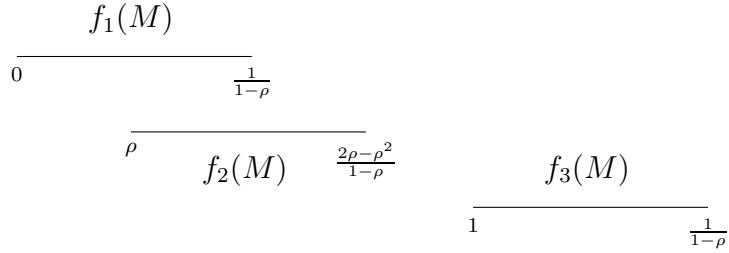
\begin{figure}[h]\label{figure1}
\centering
\begin{tikzpicture}[scale=5]
\draw(0,0)node[below]{\scriptsize $0$}--(.618,0)node[below]{\scriptsize$\frac{1}{1-\rho}$};
\draw(0.3,-0.2)node[below]{\scriptsize $\rho$}--(.918,-0.2)node[below]{\scriptsize$\frac{2\rho-\rho^2}{1-\rho}$};
\draw(1.2,-0.4)node[below]{\scriptsize $1$}--(1.818,-0.4)node[below]{\scriptsize$\frac{1}{1-\rho}$};
\node [label={[xshift=1.5cm, yshift=0cm]$f_1(M)$}] {};
\node [label={[xshift=3cm, yshift=-2cm]$f_2(M)$}] {};
\node [label={[xshift=7.5cm, yshift=-2cm]$f_3(M)$}] {};
\end{tikzpicture}\caption{The location of  $f_i(M), i=1,2,3$.}
\end{figure}
Thus one has that ${\mathcal S}_1=\{3\}$ and  ${\mathcal S}_2=\{23\}$. For $k\geq 3$ the sets ${\mathcal S}_k$ becomes a bit complicated. However, it is not so difficult to find out
that $|{\mathcal S}_k|=k-1$ by noting that  $f_1\circ f_3=f_2\circ f_1$, where $|A|$ denotes the cardinality of set $A$. Let $\Gamma =\bigcup _{k\geq 1}{\mathcal S}_k$. Thus by Lemma \ref{gamma1}
$$
\dim _H\Pi \left (\Gamma ^{\mathbb N}\right )=s, \;\textrm{where}\; \rho ^{s}+\rho ^{2s}+\sum_{k=3}^{\infty}(k-1)\rho ^{ks}=1.
$$
It is  an easy exercise to check that $s=\frac{\log \lambda}{-\log \rho}$  where $\lambda\approx 2.3247$ is the appropriate solution of $x^{3}-3x^{2}+2x-1=0.$

Now we show that $\dim _H\Pi(V)\leq s$, where $V=\{{\bf i}\in \Omega ^{\mathbb N}: {\bf i}\;\textrm{does not begin with}\; \Gamma \}$ is as that in Theorem \ref{thm1}. By the geometric structure of $K$ one can see that
for each positive integer $k$, the set $\Pi(V)$ can be covered by $2^k$ many number of intervals of length $\rho^k(\rho-1)^{-1}$. Thus
$$
\mathcal{H}_{\rho^k(\rho-1)^{-1}}^{s}(\Pi(V))\leq (1-\rho)^{-s}(2\rho^{s})^{k}\rightarrow 0\; \textrm{as}\; k\rightarrow \infty
$$
since  $2\rho^{s}<1$. Thus, $\dim_{H}U=\frac{\log \lambda}{-\log \rho}$ by Theorem \ref{thm1}.
\end{proof}

\begin{Example}\label{EX2}
Take $0<\lambda<(3-\sqrt{5})/2$.
 Let  $K$ be the self-similar set generated by the   IFS $\{f_1, \cdots , f_5\}$ where
 $$
 f_{i}(x,y)=(\lambda x, \lambda y)+(a_i, b_i)
 $$
 with $(a_1, b_1)=(0,0), (a_2, b_2)=(1-\lambda, 0), (a_3, b_3)=(1-\lambda, 1-\lambda), (a_4, b_4)=(0, 1-\lambda)$ and $
 (a_5, b_5)= (\lambda(1-\lambda), (1-\lambda)^2)$.
  Then
 $\dim_{H}U=s\approx \dfrac{\log 4.61347}{-\log \lambda}$ where $\lambda ^{3s}-2\lambda ^{2s}+5\lambda ^s-1=0$.
 \end{Example}
 \begin{proof}
 First one can check that $f_4\circ f_2=f_5\circ f_4$.
  Among the squares $f_i([0,1]^2), 1\leq i\leq 5$, only $f_4([0,1]^2)\cap f_5([0,1]^2)\ne \emptyset $ (see Figure 2).
   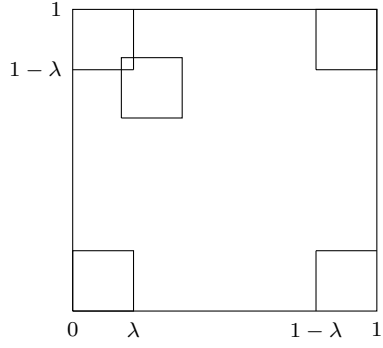
\begin{figure}[h]\label{figure2}
 \centering
\begin{tikzpicture}[scale=4]
\draw(0,0)node[below]{\scriptsize 0}--(.2,0)node[below]{\scriptsize$\lambda$}--(.8,0)node[below]{\scriptsize$1-\lambda$}--(1,0)node[below]{\scriptsize$1$}--(1,1)--(0,1)node[left]{\scriptsize$1$}--(0,.5)--(0,0);
\draw(.8,0)--(0.8,0.2);
\draw(0.8,0.2)--(1,0.2);
\draw(0,0.8)node[left]{\scriptsize $1-\lambda$}--(0.2,0.8);
\draw(0.2,0.8)--(0.2,1);
\draw(0.16,0.64)--(0.36,0.64)--(0.36,0.84)--(0.16,0.84)--(0.16,0.64);
\draw(0,0)--(0.2,0)--(0.2,0.2)--(0,0.2)--(0,0);
\draw(0.8,0.8)--(1,0.8)--(1,1)--(0.8,1)--(0.8,0.8);
\end{tikzpicture}\caption{The locations of squares $f_i([0,1]^2), i=1,2,3,4,5$}
\end{figure}
 Thus ${\mathcal S}_1=\{1,2,3\}$ and  ${\mathcal S}_2=\{41, 43, 51, 52,53\}$. As in above example, for $k\geq 3$ the sets ${\mathcal S}_k$ becomes a bit complicated. However, it is not so difficult to find out
that $|{\mathcal S}_k|=3k-1$ by noting that $f_4\circ f_2=f_5\circ f_4$.
Let $\Gamma =\bigcup _{k\geq 1}{\mathcal S}_k$. Thus by Lemma \ref{gamma1}
we have $\dim _H\Pi \left (\Gamma ^{\mathbb N}\right )=s\approx \dfrac{\log 4.61347}{-\log \lambda}$ where
$$
3\lambda^{s}+5\lambda^{2s}+\sum_{k=3}^{\infty}(3k-1)\lambda^{ks}=1,
$$
which is equivalent to $\lambda ^{3s}-2\lambda ^{2s}+5\lambda ^s-1=0$.

Now we show that $\dim _H\Pi(V\cap \Pi^{-1}(U)))\leq s$,  where $$V=\{{\bf i}\in \Omega ^{\mathbb N}: {\bf i}\;\textrm{does not begin with}\; \Gamma \}$$ is as that in Theorem \ref{thm1}. By the geometric structure of $K$ one can see that
for each positive integer $k$, the set $\Pi(V\cap \Pi^{-1}(U)))$ can be covered by $2^k$ many number of  squares with diameter $\sqrt{2}\lambda^k$. Thus
$$
\mathcal{H}_{\sqrt{2}\lambda^k}^{s}(K_{\alpha})\leq  2^k\sqrt{2}^{s}\lambda^{sk} \rightarrow 0\; \textrm{as}\; k\rightarrow \infty
$$
since $2\lambda^{s}<1$. Thus, $\dim_{H}U=s$ by Theorem \ref{thm1}.
 \end{proof}
 In the above we change the map $f_5$ by letting
 $$
(a_5, b_5)=(\lambda-\lambda^{u+1}, 1-2\lambda+\lambda^{u+1})\;\;\textrm{with}\;\;u\in \mathbb N,
$$
where we require that $\lambda ^{u+1}-3\lambda +1>0$. Then $\dim _HU$ can be also obtained by the same way as in Example \ref{EX2}
and so $\dim _HU_k, k\geq 2$ can be obtained as well. In fact, we have
\begin{Example}\label{EX3}
Suppose that $\lambda \in (0,1), u\in \mathbb N$ satisfy $\lambda ^{u+1}-3\lambda +1>0$.
 Let  $K$ be the self-similar set generated by the   IFS $\{f_1, \cdots , f_5\}$ where
 $$
 f_{i}(x,y)=(\lambda x, \lambda y)+(a_i, b_i)
 $$
 with $(a_1, b_1)=(0,0), (a_2, b_2)=(1-\lambda, 0), (a_3, b_3)=(1-\lambda, 1-\lambda), (a_4, b_4)=(0, 1-\lambda)$ and $
 (a_5, b_5)=(\lambda-\lambda^{u+1}, 1-2\lambda+\lambda^{u+1})$.
  Then
 $$
 \dim_{H}U_{k+1}=\dim_{H}U
 \;\;\textrm{for any }\;\; k\in \mathbb N.
 $$
 \end{Example}
\begin{proof}
By (\ref{ksetine}) we only need to show that $\dim _HU_{k+1}\geq \dim _HU$. This will be done by showing
\begin{equation*}\label{lll}
f_{42^{uk}1}(U)\subseteq U_{k+1}\;\; \textrm{for each }\;\;k\geq 1.
\end{equation*}
 Now arbitrarily fix a point $c\in U$ with the unique coding $(c_i)$.
 We prove that $f_{42^{uk}1}(c)\in U_{k+1}, k\geq 1$
 by induction.

Let $k=1$.
 Note that $x_1=f_{42^u1}(c)=f_{54^u1}(c)\in f_{42^u}([0,1]^2)= f_{54^u}([0,1]^2)$
 and
 $$
 f_{\bf i}([0,1]^2)\cap f_{42^u}([0,1]^2)=\emptyset \;\;\textrm {for all}\;\;{\bf i}\in
 \{1,2,3,4,5\}^{u+1}\setminus \{42^u, 54^u, 54^{u-1}5\}.
 $$
 Hence any coding $(d_i)$ of $x_1$ has to begin with $42^u $, $54^u$ or $54^{u-1}5$. We claim that
 $(d_i)$ cannot begin with $54^{u-1}5$. Otherwise, we have $f_{41} (c)=f_{51}(c)\in f_{41}([0,1]^2)\cap f_{51}([0,1]^2)=\emptyset $.
 On the other hand,
 we have
$\pi (\sigma ^{u+1}(d_i))=\pi (1(c_i))=f_1(c)\in U$ where $\sigma $ is the left shift on $\Omega ^{\mathbb N}$. Thus, $(d_i)$ has to be $42^u1(c_i)$ or
 $54^u1(c_i)$, i.e., $x_1\in U_2$.

 Suppose that $x_k=f_{42^{uk}1}(c)=\pi(42^{uk}1(c_i))\in U_{k+1}$. Let $(d_i)$ be a coding of $x_{k+1}:=f_{42^{u(k+1)}1}(c)$. As before we know that $(d_i)$ has to begin with $42^u $, $54^u$ or $54^{u-1}5$, and so
 $(d_i)$ has to begin with $42^{u-1} $ or $54^{u-1}$.
   Note that
$$x_{k+1}=f_{42^u}(\pi(2^{uk}1(c_i)))=f_{54^u}(\pi(2^{uk}1(c_i)))
=f_{54^{u-1}}(\pi(42^{uk}1(c_i)))=f_{54^{u-1}}(x_k).
$$
For the case that $(d_i)$  begins with $42^{u-1} $ we have that $(d_i)=42^{uk}1(c_i)$ since
$\pi (\sigma ^{u}(d_i))=\pi (2^{uk+1}1(c_i))\in U$. For the case that
$(d_i)$  begins with $54^{u-1}$ we have that $(d_i)$ has exactly $k+1$ many choices since
$$
\pi ((d_i)_{i>u})=\pi (\sigma ^{u}(d_i))=\pi(42^{uk}(c_i))=x_k \in U_{k+1}.
$$
Hence we complete the proof.
\end{proof}

  In the last example we try to describe $\Gamma $ by a way which was  developed in \cite{NW, LauNgai}.
  \begin{Example}
  Let $K$ be the self-similar set generated by the IFS
  $$\left\{f_1(x)=\dfrac{x}{4},\; f_2(x)=\dfrac{x}{4}+\dfrac{9}{17},\; f_3(x)=\dfrac{x+3}{4}\right\}$$
  Then $\dim_{H}U=s,$ where $s$ is the unique solution of the following equation:
  $$
  \dfrac{1}{4^s}+\dfrac{1}{4^{2s}}+\sum_{n=2}^{\infty}(a_n+c_n)\dfrac{1}{4^{(n+1)s}}=1,
  $$
  where $a_n, c_n$ for $n\geq 2$ are determined by
  \begin{equation}\label{numbercou}
  \left (
  \begin{array}{l}
  a_n   \\
  b_n  \\
  c_n  \\
  d_n  \\
  e_n  \\
  \end{array}
  \right )=
  \left (
  \begin{array}{lllll}
  0 & 1 & 0 & 0& 1   \\
  0 & 1 & 0 & 1& 1  \\
1 & 0 & 1 & 0& 0  \\
  1 & 0 & 1& 0& 0  \\
  0 & 1 & 0 & 1& 0   \\
  \end{array}
  \right )^{n-2}\left (
  \begin{array}{l}
  1   \\
  1  \\
  1 \\
  1  \\
  1 \\
  \end{array}
  \right )
  \end{equation}
  \end{Example}
  \begin{proof}
  We take $M=[0,1]$ (one can check that $f_i(M)\subseteq M$ for $i\in \Omega :=\{1,2,3\}$) and label it by $T_1$. Its offspring are
  $$
  f_1(M)=[0, 1/4],\;\; f_2(M)=[9/17, 53/68]\;\textrm{and}\;\; f_3(M)=[3/4, 1].
  $$
   Then (see Figure 3)
  $$
  {\mathcal S}_1=\{1\}\;\;\textrm{and }\; {\mathcal T}_1=\{2,3\}.
  $$
  \begin{figure}[h]\label{figure3}
\centering
\begin{tikzpicture}[scale=5]
\draw(0,0)node[below]{\scriptsize $0$}--(.618,0)node[below]{\scriptsize$1/4$};
\draw(0.9,-0.2)node[below]{\scriptsize $9/17$}--(1.518,-0.2)node[below]{\scriptsize$53/68$};
\draw(1.2,-0.4)node[below]{\scriptsize $3/4$}--(1.818,-0.4)node[below]{\scriptsize$1$};
\node [label={[xshift=1.5cm, yshift=0cm]$f_1(M)$}] {};
\node [label={[xshift=6cm, yshift=-1cm]$f_2(M)$}] {};
\node [label={[xshift=7.5cm, yshift=-3cm]$f_3(M)$}] {};
\end{tikzpicture}\caption{The location of  $f_i(M), i=1,2,3$.}
\end{figure}
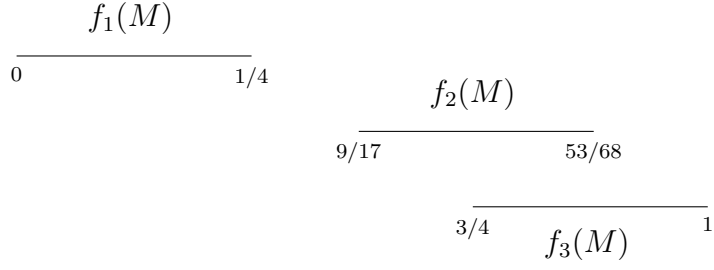

  Note that the offspring of $f_1(M)$ have the same geometric location as the offspring of $M$. So $f_1(M)$
is labeled by $T_1$ as well. We label $f_2(M)$ and $f_3(M)$ by $T_2$ and $T_3$, respectively. Thus one can
 simply denote $M$ and its offspring as follows:
$$
(M, T_1)\rightarrow (f_1(M), T_1)+( f_2(M) T_2)+( f_3(M), T_3).
$$

Now let us calculate $f_{\bf i}(M): {\bf i}\in {\mathcal T}_1\times \Omega $:
\begin{equation*}
\begin{array}{lll}
f_{21}(M)=\left[\frac{9}{17}, \frac{9}{17}+\frac{1}{16} \right], & f_{22}(M)=\left[\frac{45}{68}, \frac{45}{68}+\frac{1}{16}\right], &f_{23}(M)=\left[\frac{195}{272}, \frac{195}{272}+\frac{1}{16}\right]\\
f_{31}(M)=\left[\frac{3}{4}, \frac{3}{4}+\frac{1}{16}\right], &f_{32}(M)=\left[\frac{60}{68}, \frac{60}{68}+\frac{1}{16}\right], &f_{33}(M)=\left[\frac{15}{16}, 1\right].
 \end{array}
 \end{equation*}
 Thus we have (see Figure 4)
 $$
  {\mathcal S}_2=\{21\}\;\;\textrm{and }\; {\mathcal T}_2=({\mathcal T}_1\times \Omega )\setminus {\mathcal S}_2
  =
  \{22, 23, 31, 32, 33\}.
  $$
  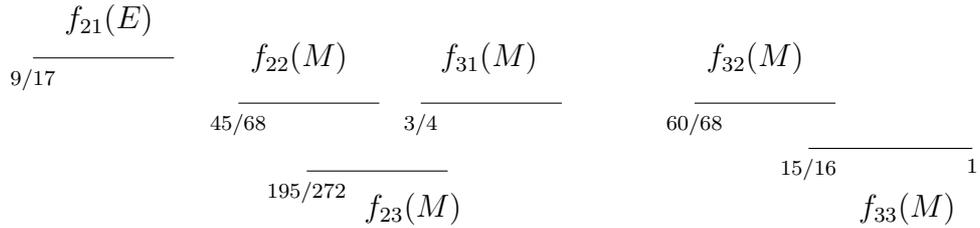
\begin{figure}[h]\label{figure4}
\centering
\begin{tikzpicture}[scale=3]
\draw(0,0)node[below]{\scriptsize $9/17$}--(.618,0)node[below]{};
\draw(0.9,-0.2)node[below]{\scriptsize $45/68$}--(1.518,-0.2)node[below]{};
\draw(1.2,-0.5)node[below]{\scriptsize $195/272$}--(1.818,-0.5)node[below]{};
\draw(1.7,-0.2)node[below]{\scriptsize $3/4$}--(2.318,-0.2)node[below]{};
\draw(2.9,-0.2)node[below]{\scriptsize $60/68$}--(3.518,-0.2)node[below]{};
\draw(3.4,-0.4)node[below]{\scriptsize $15/16$}--(4.118,-0.4)node[below]{\scriptsize $1$};
\node [label={[xshift=1cm, yshift=0cm]$f_{21}(E)$}] {};
\node [label={[xshift=3.5cm, yshift=-0.5cm]$f_{22}(M)$}] {};
\node [label={[xshift=5cm, yshift=-2.5cm]$f_{23}(M)$}] {};
\node [label={[xshift=6cm, yshift=-0.5cm]$f_{31}(M)$}] {};
\node [label={[xshift=9.5cm, yshift=-0.5cm]$f_{32}(M)$}] {};
\node [label={[xshift=11.5cm, yshift=-2.5cm]$f_{33}(M)$}] {};
\end{tikzpicture}\caption{The location of  $f_{\bf i}(M), {\bf i}\in {\mathcal T}_1\times \Omega $}
\end{figure}
  By the same argument as above the offspring $f_{21}(M)$ of $f_2(M)$ has label $T_1$,  while the other two offspring $f_{22}(M)$, $f_{23}(M)$ of $f_2(M)$ will obtain new labels $T_4, T_5$, respectively. This can be simply denoted by
  \begin{equation}\label{ttwo}
  (f_2(M), T_2)\to (f_{21}(M) ,T_1)+(f_{22}(M), T_4)+(f_{23}(M) ,T_5).
  \end{equation}
  Similarly, for the $f_3(M)$ and its offspring we have
\begin{equation}\label{tthree}
(f_3(M), T_3)\to (f_{31}(M), T_6)+ (f_{32}(M), T_2)+(f_{33}(M), T_3).
\end{equation}
Since one knows what will happen for the offspring of $f_{32}(M)$ and $f_{33}(M)$, let us  continue to calculate the offspring of $f_{\bf i}(M): {\bf i}\in {\mathcal T}_2\setminus \{32,33\}=\{22,23,31\} $:
  \begin{equation*}
\begin{array}{lll}
f_{221}(M)=\left[\frac{45}{68},  \frac{45}{68}+\frac{1}{64}\right], &
f_{222}(M)=\left[\frac{189}{272}, \frac{189}{272}+\frac{1}{64}\right], & f_{223}(M)=\left[\frac{771}{1088}
 \frac{771}{1088}+\frac{1}{64}\right]\\
 f_{231}(M)=\left[\frac{195}{272},
  \frac{195}{272}+\frac{1}{64}\right], &
   f_{232}(M)=\left[\frac{51}{68}, \frac{51}{68}+\frac{1}{64}\right], &
  f_{233}(M)=\left[\frac{831}{1088}, \frac{831}{1088}+\frac{1}{64}\right]\\
 f_{311}(M)=\left[\frac{3}{4}, \frac{3}{4}+\frac{1}{64}\right], &
 f_{312}(M)=\left[\frac{213}{272}, \frac{213}{272}+\frac{1}{64}\right], & f_{313}(M)=\left[\frac{51}{64}, \frac{52}{64}\right].
 \end{array}
 \end{equation*}
 Thus we have (see Figure 5)
 \begin{equation}\label{seclev}
 \begin{split}
 &(f_{22}(M), T_4)\rightarrow (f_{221}(M), T_1)+(f_{222}(M), T_4)+(f_{223}(M), T_5)\\
 &(f_{23}(M) ,T_5)\rightarrow  (f_{231}(M) ,T_6)+(f_{232}(M) ,T_2)+(f_{233}(M) ,T_3)\\
 &(f_{31}(M), T_6)\rightarrow  (f_{311}(M), T_2)+(f_{312}(M), T_2)+(f_{313}(M), T_3)
 \end{split}
 \end{equation}
 \begin{figure}[h]\label{figure5}
\centering
\begin{tikzpicture}[scale=2]
\draw(0,0)node[below]{\scriptsize $\dfrac{45}{68}$}--(.618,0)node[below]{};
\draw(0.9,-0.2)node[below]{\scriptsize $\dfrac{189}{272}$}--(1.518,-0.2)node[below]{};
\draw(1.2,-0.7)node[below]{\scriptsize $\dfrac{771}{1088}$}--(1.818,-0.7)node[below]{};
\draw(1.7,-0.2)node[below]{\scriptsize $\dfrac{195}{272}$}--(2.318,-0.2)node[below]{};
\draw(2.9,-0.2)node[below]{\scriptsize $\frac{51}{68}$}--(3.518,-0.2)node[below]{};
\draw(3.4,-0.4)node[below]{\scriptsize $\frac{831}{1088}$}--(4.118,-0.4)node[below]{};
\draw(4.9,-0.2)node[below]{\scriptsize $\frac{213}{272}$}--(5.518,-0.2)node[below]{};
\draw(5.4,-0.4)node[below]{\scriptsize $\frac{51}{64}$}--(6.118,-0.4)node[below]{\scriptsize $\frac{52}{64}$};
\node [label={[xshift=1cm, yshift=0cm]$f_{221}(M)$}] {};
\node [label={[xshift=2.5cm, yshift=-0.5cm]$f_{222}(M)$}] {};
\node [label={[xshift=4cm, yshift=-0.5cm]$f_{231}(M)$}] {};
\node [label={[xshift=6.8cm, yshift=-0.5cm]$f_{232}(M)=f_{311}(M)$}] {};
\node [label={[xshift=7.5cm, yshift=-2.2cm]$f_{233}(M)$}] {};
\node [label={[xshift=3.5cm, yshift=-2.5cm]$f_{223}(M)$}] {};
\node [label={[xshift=10.5cm, yshift=-0.5cm]$f_{312}(M)$}] {};
\node [label={[xshift=11.5cm, yshift=-2.2cm]$f_{313}(M)$}] {};
\end{tikzpicture}\caption{The location of  $f_{\bf i}(M), {\bf i}\in ({\mathcal T}_2\setminus \{32,33\})\times \Omega $}
\end{figure}
It is important to notice that no more labels occur in the above expression.
Note that we have $f_{232}=f_{311}$. Thus $f_{232}(M)$ and $f_{311}(M)$ contribute nothing to $\Gamma $. Therefore, we replace (\ref{seclev}) by
\begin{equation}\label{seclevv}
 \begin{split}
 &(f_{22}(M), T_4)\rightarrow (f_{221}(M), T_1)+(f_{222}(M), T_4)+(f_{223}(M), T_5)\\
 &(f_{23}(M) ,T_5)\rightarrow  (f_{231}(M) ,T_6)+(f_{233}(M) ,T_3)\\
 &(f_{31}(M), T_6)\rightarrow  (f_{312}(M), T_2)+(f_{313}(M), T_3)
 \end{split}
 \end{equation}

It follows from (\ref{ttwo}), (\ref{tthree}) and
(\ref{seclevv}) that only $T_2$ and $T_4$ have contribution to $\Gamma $. Together with (\ref{Sdef}) one knows that the cardinality of ${\mathcal S}_{n+1}$ ($n\geq 2$) equals to the number of $T_2$ and $T_4$ occurring in the
$n$-th generation offspring. By $a_n, b_n, c_n, d_n$ and $e_n$ we denote the number of $T_2, T_3, T_4, T_5$ and $T_6$ occurring in the $n$-th generation offspring. By (\ref{ttwo}), (\ref{tthree}) and
(\ref{seclevv}) we have
\begin{equation*}
  \left (
  \begin{array}{l}
  a_{n+1}   \\
  b_{n+1}  \\
  c_{n+1}  \\
  d_{n+1}  \\
  e_{n+1}  \\
  \end{array}
  \right )=
  \left (
  \begin{array}{lllll}
  0 & 1 & 0 & 0& 1   \\
  0 & 1 & 0 & 1& 1  \\
1 & 0 & 1 & 0& 0  \\
  1 & 0 & 1& 0& 0  \\
  0 & 1 & 0 & 1& 0   \\
  \end{array}
  \right )
   \left (
  \begin{array}{l}
  a_n   \\
  b_n  \\
  c_n  \\
  d_n  \\
  e_n  \\
  \end{array}
  \right ):=A\left (
  \begin{array}{l}
  a_n   \\
  b_n  \\
  c_n  \\
  d_n  \\
  e_n  \\
  \end{array}
  \right ),\;\; n\geq 2.
  \end{equation*}

By (\ref{ttwo}) and (\ref{tthree}) we have
$$
a_2=b_2=c_2=d_2=e_2=1,
$$
and so (\ref{numbercou}) is obtained.
Therefore, we have $\dim _H\Pi \left (\Gamma ^{\mathbb N}\right )=s$, where $s$ is the unique solution of the following equation:
  $$
 1=\dfrac{1}{4^s}+\dfrac{1}{4^{2s}}+\sum_{n=2}^{\infty}|{\mathcal S}_{n+1}|\dfrac{1}{4^{(n+1)s}}= \dfrac{1}{4^s}+\dfrac{1}{4^{2s}}+\sum_{n=2}^{\infty}(a_n+c_n)\dfrac{1}{4^{(n+1)s}}.
  $$

In what follows we will show that $\dim _H\Pi (V\cap \Pi^{-1}(U))\leq s$.
One can check that the spectral radius of $A$ is about $\lambda \approx2.2775$. We claim that $\lambda <4^s$.
In fact, we have $4^s>4^t\approx 2.4693$ where $t$ is determined by
$$
\frac{1}{4^t}+\frac{1}{4^{2t}}+\sum_{n=2}^5(a_n+c_n)\frac{1}{4^{(n+1)t}}=
\frac{1}{4^t}+\frac{1}{4^{2t}}+\frac{2}{4^{3t}}+\frac{4}{4^{4t}}+\frac{9}{4^{5t}}+\frac{21}{4^{6t}}
=1.
$$
 Note that
 $$
 \lim _{n\rightarrow \infty} \mathcal{H}^{s}_{4^{-n-2}}(\Pi (V\cap \Pi^{-1}(U)))\leq \overline{\lim}_{n\rightarrow \infty }(a_{n+1}+b_{n+1}+c_{n+1}+d_{n+1}+e_{n+1})4^{(-n-2)s}<\infty ,
 $$
where the last inequality holds since all the
$a_{n+1},b_{n+1},c_{n+1},d_{n+1},e_{n+1}$ are bounded by $c\lambda ^n$ for some $c>0$, and the fact
$\lambda <4^s$.
\end{proof}

\section*{Acknowledgment}
The first author is granted by the China Scholarship Council grant No. 201606140059. The third author  was
supported by  NSFC No. 11671147, 11571144 and Science and Technology Commission of Shanghai Municipality (STCSM) grant No. 13dz2260400.


\begin{thebibliography}{10}

\bibitem{BakerG}
Simon Baker.
\newblock Generalized golden ratios over integer alphabets.
\newblock {\em Integers}, 14:Paper No. A15, 28, 2014.

\bibitem{SKK}
Simon Baker, Karma Dajani, and Kan Jiang.
\newblock On univoque points for self-similar sets.
\newblock {\em Fund. Math.}, 228(3):265--282, 2015.


\bibitem{DJKL}
Karma Dajani, Kan Jiang, Derong Kong, and Wenxia Li.
\newblock Multiple expansions of real numbers with digits set $\{0,1,q\}$.
\newblock {\em arXiv:1508.06138}, 2015.

\bibitem{KarmaKan2}
Karma Dajani, Kan Jiang, Derong Kong, and Wenxia Li.
\newblock Multiple codings for self-similar sets with overlaps.
\newblock {\em 	arXiv:1603.09304}, 2016.


\bibitem{MK}
Martijn de~Vries and Vilmos Komornik.
\newblock Unique expansions of real numbers.
\newblock {\em Adv. Math.}, 221(2):390--427, 2009.



\bibitem{GS}
Paul Glendinning and Nikita Sidorov.
\newblock Unique representations of real numbers in non-integer bases.
\newblock {\em Math. Res. Lett.}, 8(4):535--543, 2001.


\bibitem{Hutchinson}
John~E. Hutchinson.
\newblock Fractals and self-similarity.
\newblock {\em Indiana Univ. Math. J.}, 30(5):713--747, 1981.


\bibitem{KarmaKan}
 K. Jiang and K. Dajani.
\newblock Subshifts of finite type and self-similar sets.
\newblock {\em Nonlinearity}, 30(2):659--686, 2017.

\bibitem{KO}
Vilmos Komornik.
\newblock Expansions in noninteger bases.
\newblock {\em Integers}, 11B:Paper No. A9, 30, 2011.

\bibitem{KKLL}
Vilmos Komornik, Derong Kong, and Wenxia Li.
\newblock Hausdorff dimension of univoque sets and devil's staircase.
\newblock {\em  Adv. in Math.}, 305(2017), 165--196, 2016.

\bibitem{KDLW}
Derong Kong and Wenxia Li.
\newblock Hausdorff dimension of unique beta expansions.
\newblock {\em Nonlinearity}, 28(1):187--209, 2015.

\bibitem{LauNgai}
Ka-Sing Lau and Sze-Man Ngai.
\newblock A generalized finite type condition for iterated function systems.
\newblock {\em Adv. in Math.}, 208(2):647--671, 2007.




 \bibitem{NW}
Ngai S M, Wang Y.  \newblock Hausdorff dimension of overlapping self-similar sets. \newblock {\em J London Math Soc.}, 63(2): 655-672, 2001.

\bibitem{Schief}
Andreas Schief.
\newblock Separation properties for self-similar sets.
\newblock {\em Proc. Amer. Math. Soc.}, 122(1):111--115, 1994.

\bibitem{Sidorov}
Nikita Sidorov.
\newblock Almost every number has a continuum of {$\beta$}-expansions.
\newblock {\em Amer. Math. Monthly}, 110(9):838--842, 2003.

\bibitem{SN}
Nikita Sidorov.
\newblock Expansions in non-integer bases: lower, middle and top orders.
\newblock {\em J. Number Theory}, 129(4):741--754, 2009.

\bibitem{ZL}
Yuru Zou, Jian Lu, and Wenxia Li.
\newblock Unique expansion of points of a class of self-similar sets with
  overlaps.
\newblock {\em Mathematika}, 58(2):371--388, 2012.

\bibitem{ZYL}
Yuru Zou, Yuanyuan Yao, and Wenxia Li.
\newblock A class of {S}ierpinski carpets with overlaps.
\newblock {\em J. Math. Anal. Appl.}, 340(2):1422--1432, 2008.

\end{thebibliography}
\end{document}